\documentclass[sigconf,authorversion=true, nonacm=true]{acmart}
\usepackage{caption}
\usepackage{subcaption}
\usepackage[ruled, vlined]{algorithm2e}
\usepackage{tabularx}

\newcommand{\REM}[1]{}
\usepackage{soul}

\newcommand{\mc}{MC\xspace} 
\newcommand{\mt}{MP\xspace} 

\usepackage{xcolor}

\usepackage{soulpos}
\ulposdef{\hlc}[xoffset=1pt]{\mbox{\color{cyan!30}\rule[-.8ex]{\ulwidth}{3ex}}}

\SetKwComment{Comment}{/* }{ */}

\theoremstyle{remark}
\newtheorem*{remark}{Remark}

\newtheorem{proposition}{Proposition}

\begin{document}

\title{The Ensemble Approach of Column Generation for Solving Cutting Stock Problems}


\author{Mingjie Hu}
\authornote{Contributed during internship at Microsoft Research.}
\email{mingjie@fudan.edu}
\affiliation{%
  \institution{Fudan University}
  \city{Shanghai}
  \country{China}
}

\author{Jie  Yan}
\email{yanjie@ict.ac.cn}
\authornote{Contributed during full-time at Microsoft Research.}
\affiliation{%
  \institution{Microsoft Research}
  \city{Beijing}
  \country{China}
}

\author{Liting Chen}
\authornotemark[1]
\email{98chenliting@gmail.com}
\affiliation{%
  \institution{Microsoft}
  \city{Beijing}
  \country{China}
}

\author{Qingwei  Lin}
\email{qlin@microsoft.com}
\affiliation{%
  \institution{Microsoft}
  \city{Beijing}
  \country{China}
}

\author{Dongmei Zhang}
\email{dongmeiz@microsoft.com}
\affiliation{%
  \institution{Microsoft}
  \city{Beijing}
  \country{China}
}

\renewcommand{\shortauthors}{Mingjie Hu and Jie Yan, et al.}

\begin{abstract}
This paper investigates the column generation (CG) for solving cutting stock problems (CSP). Traditional CG method, which repeatedly solves a restricted master problem (RMP), often suffers from two critical issues in practice -- the loss of solution quality introduced by linear relaxation of both feasible domain and objective and the high time cost of last iterations close to convergence. We empirically find that the first issue is common in ordinary CSPs with linear cutting constraints, while the second issue is especially severe in CSPs with nonlinear cutting constraints that are often generated by approximating chance constraints. We propose an alternative approach, ensembles of multiple column generation processes. In particular, we present two methods -- \mc (multi-column) which return multiple feasible columns in each RMP iteration, and \mt (multi-path) which restarts the RMP iterations from different initialized column sets once the iteration time exceeds a given time limit. The ideas behind are same: leverage the multiple column generation pathes to compensate the loss induced by relaxation, and add earlier sub-optimal columns to accelerate convergence of RMP iterations. Besides, we give theoretical analysis on performance improvement guarantees. Experiments on cutting stock problems demonstrate that compared to traditional CG,  our method achieves significant run-time reduction on CSPs with nonlinear constraints, and dramatically improves the ratio of solve-to-optimal on CSPs with linear constraints.
\end{abstract}




\maketitle

\section{Introduction}
\label{sec:introduction}
The Cutting Stock Problem (CSP)\cite{ben2005cutting} is one classic problem in operational research with widespread practical applications, such as production planning\cite{kantorovich1960mathematical} and resource allocation\cite{Katoh1998}. Although in terms of computational complexity, CSP is a NP-hard problem, it can be formulated as an integer linear programming problem and solved with Column Generation (CG) method \cite{gilmore1961csp, gilmore1963csp2}. 
CG involves iteratively generating columns to refine an initially feasible solution until an optimal solution is found, making it possible to solve problems with a large number of variables and constraints. 
Despite the great success of CG on solving CSP, there are still two issues regarding this method that are not widely addressed. The first is the gap between CSP and the relaxed linear programming problem during the iteractive process of generating columns, which can lead to a decrease of solution quality. The second is the difficulty in managing the solving time for problems with nonlinear constraints, even for moderate-sized problems. In real-world applications, these nonlinear constraints are introduced because the size of the items in the CSP can be random variables. For instance, in cloud computing, the peak resource usage of a container is a random variable determined by its assigned tasks. When dealing with such nonlinear constraints, the standard column generation method, which solves a relaxed LP problem can be time-consuming and may compromise the quality of the columns generated.

The aim of our paper is driven by two objectives. The first objective aims to address the sub-optimality caused by the gap between integer and linear programming relaxations by incorporating columns generated from diverse solving paths. The second objective is critical in real-world applications where decisions need to be made promptly. To achieve this, our strategy involves including high-quality and feasible columns in the column set at an early stage, even if they are sub-optimal with respect to the column generation iterations of the Restricted Master Problem. We propose the ensemble method for column generation. To the best of our knowledge, this is the first such work in the area.
 Our main contributions are as follows.
\begin{itemize}
    \item Empirically identify the solution quality issue and time cost issue of traditional column generation methods on solving CSP with linear and nonlinear cutting constraints. 
    \item Propose the ensemble approach of column generation, \mc (multi-column) and \mt (multi-path) methods, and give deatiled algorithms and a comprehensive analysis.
    \item Demonstrate that the proposed ensemble methods achieves significant run-time reduction on CSPs with nonlinear cutting constraints, and dramatically improves the ratio of solve-to-optimal on CSPs with linear cutting constraints at the cost of slightly more time cost.
\end{itemize}

\section{Preliminaries}
\label{sec:preliminaries}
In this section, we first introduce the CSP in the Gilmore-Gomory formulation,  and then discuss the process of solving the CSP by CG method. Besides, we describe the scenario where the column has nonlinear cutting constraints.

\subsection{Cutting Stock Problem}
Without loss of generality, we mainly consider the one-dimensional CSP. It aims to satisfy the rolls demand $b_{k}, k\in [K]$ of $K$ customers for rolls of length $v_{k}$ by cutting the least amount of rolls of length $V$ \cite{vance1998branch}.
We are interested in the Gilmore-Gomory formulation \cite{gilmore1961csp, gilmore1963csp2} of CSP, which contains so many columns that it can only be solved efficiently by the CG method.

\begin{equation}
\begin{gathered}
\min _{\boldsymbol{x}} \sum_{p \in \mathcal P} x_p \\
\text { s.t. } \sum_{p \in \mathcal P} a_{k p} x_p \geq b_k, \forall k \in [K] \\
x_p \in I, \forall p \in \mathcal P
\end{gathered}
\label{csp}
\end{equation}
where $a_{p}=(a_{1p},a_{2p},...,a_{Kp})^{T}$ denotes a feasible cutting pattern, $\mathcal{P}$ is the set of the index of all feasible patterns. The decision variable $x_{p}$ represents the number of times pattern $a_{p}$ is used. A feasible cutting pattern satisfies the following constraints:

\begin{equation}
\begin{aligned}
&\sum_{k \in[K]} a_{k p} v_k \leq V \\
&a_{k p} \in I, \forall k \in[K]
\end{aligned}
\label{eq:column_linear}
\end{equation}

The LP relaxation of formulation (\ref{csp}) provides a tighter lower bound on the optimal integer solution. However, there are so many columns in (\ref{csp}) that it is impractical to enumerate all the columns and solve them with integer programming method.

\subsection{Column Generation}
CG (Column Generation) is an algorithm specifically designed to solve large-scale linear programming (LP) problems. We consider to solve the LP relaxation of formulation~(\ref{csp}), called the master problem (MP) in CG. Instead of enumerating all columns at once, CG starts by solving the restricted master problem (RMP), a LP problem initialized with a small fraction of feasible columns. The optimal dual solution to the RMP is applied to price the columns not included in the RMP. To find feasible columns with negative reduced costs, we need to solve the following subproblem at each iteration, also known as the pricing problem (PP) in the literature.

\begin{equation}
\begin{aligned}
&\min _{\boldsymbol{a}}  1-\sum_{k \in[K]} \bar{\pi}_k a_k \\
&\text { s.t. } \sum_{k \in[K]} a_k v_k \leq V \\
&a_k \in I, \forall k \in[K]
\end{aligned}
\label{subproblem}
\end{equation}
where $\bar{\pi}=\{\pi_{1},...,\pi_{K}\}$ is the optimal dual solution to the RMP. Subproblem (\ref{subproblem}) provides a feasible column $a^{*}=(a_{1},a_{2},...,a_{K})^T$ with a minimal reduced cost. If the reduced cost is negative, we add this column to the RMP and then reoptimize RMP to get a new optimal dual solution. Otherwise, we can justify that the current solution is optimal for the MP. Finally, by limiting the variables of RMP to integers and solving the integer programming problem (IRMP), we obtain a feasible integer solution, which provides an upper bound for the optimal solution.

\subsection{Column with Nonlinear Constraints}
\label{sub:col_nonlinear}

In real-world applications, the item size $v_k$ is often a random variable, such that a feasible pattern often requires satisfying nonlinear constraints. For example, suppose $v_k \sim \mathcal{N}(\mu_k, \sigma_k)$, the domain of feasible patterns is defined as the $\{a_p\}$ where $a_p$ satisfies the following constraints.

\begin{equation}
\begin{aligned}
&\sum_{k\in[K]} \mu_k a_{k p}+D(\alpha) \sqrt{\sum_{k\in[K]}\sigma_k^2 a_{k p}} \leq V \\
&a_{k p} \in I, \forall k \in[K]
\end{aligned}
\label{eq:column_nonlinear}
\end{equation}

\section{Motivations}
\label{sec:motivation}

In practice, we found that (1) for classic cutting stock problems with linear constraints, the linear relaxation used during column generation often leads to sub-optimal results for the original CSP; (2) for cutting stock problems with nonlinear constraints, the solving time can become prohibitively large, even for middle-sized problems (e.g., $K=20$), making it unsuitable for real applications. These observations have motivated our work, which aims to address these issues and improve the performance of the CG method for CSP. In this section, we empirically analyze these two issues and further elaborate the motivations behind our work.

\subsection{Solving Time Issue}

We focus on a real-world cutting stock problem -- container scheduling in cloud\cite{yan2022solving}, which is CSP with nonlinear cutting containers. Suppose there are $K=20$ services, and each service has a number of containers ($b_k$ for the $k$th services) whose peak resource usage is modeled as a random variable with a Gaussian distribution, i.e., $v_k \sim \mathcal{N}(\mu_k, \sigma_k)$. More details on the problem setup is in Appendix~\ref{app:experiemnts}. 

\begin{figure}[h]
  \centering
  \includegraphics[width=0.45\textwidth]{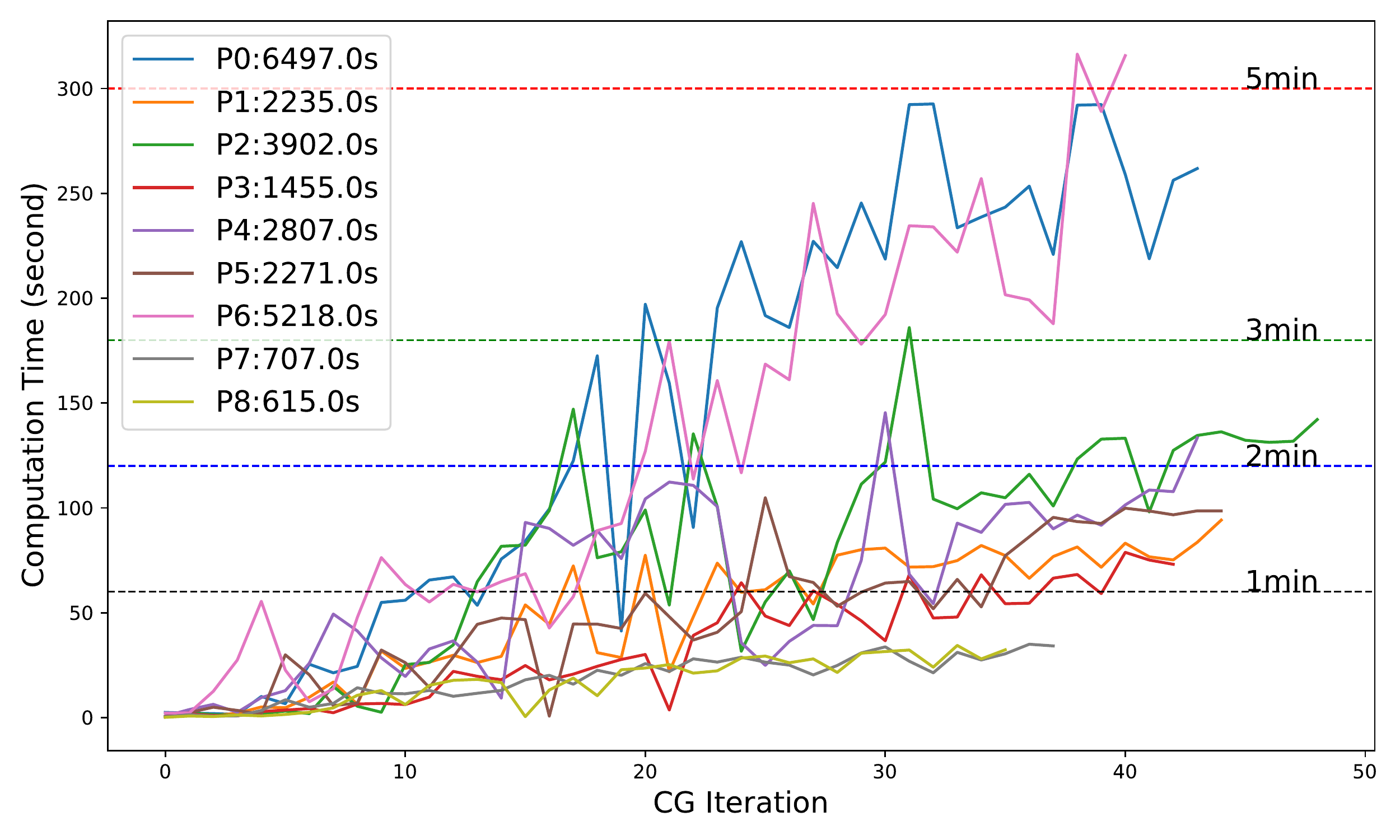}
  \caption{Runtime of Column Generation on solving 9 instances of container scheduling with nonlinear constraints}
  \label{fig:rt_baseline}
\end{figure}

The runtime of traditional column generation method on 9 problem instances are plotted in Fig.\ref{fig:rt_baseline}.The acceptable solving time for this problem is considered to be no more than 1 hour in practical terms. However, the conventional column generation method may sometimes experience an uncontrolled increase in the time it takes to find the optimal solution. In the experiments, three problems (P0, P2 and P6) take longer than one hour to solve, which is an unreasonable amount of time for real-world applications.

There are three main reasons for the long running time of the nonlinear CSP as shown in Fig.\ref{fig:rt_baseline}. Firstly, the CG sub-problem (\ref{subproblem}) is a bin-packing problem, which is an NP-hard optimization problem. Secondly, in the container scheduling example, we need to deal with nonlinear constraints (\ref{eq:column_nonlinear}), which increases the difficulty of the problem. Finally, the problem size $K$ and parameters $v_{k}$ affect the solving time. Let $\hat{v} = \min v_{k}$, then the number of items in a single pattern is bounded by $N = \lfloor \frac{V}{\hat{v}} \rfloor$. Since there are $K$ different types of items, the number of feasible patterns is bounded by $C_{N + K} ^ {N}$. Therefore, as $K$ increases or $\hat{v}$ decreases, the number of feasible patterns increases rapidly and the search space becomes larger, making it more difficult to solve.

\subsection{Solution Quality Issue}

Solution quliaty loss is induced by relaxation during column searching. Let $Z_{RMP}$ denote the optimal value of the RMP, $Z_{IRMP}$ represent the optimal value obtained by CG, $Z_{IP}^{*}$ denote the optimal value of problem (\ref{csp}). According to proposition \ref{bound}, we know that $Z_{RMP} \leq Z_{IP}^{*} \leq Z_{IRMP}$. Define the absolute gap $\delta$ equal to $Z_{IRMP} - Z_{RMP}$. Since $Z_{IP}^{*}$ and $Z_{IRMP}$ are integer, $\delta < 1$ means that $Z_{IP}^{*} = Z_{IRMP}$ and we find the optimal solution, otherwise, there may be room for improvement in the solution quality. 

Let's look into the experimental results of ordinary CSPs with linear cutting constraints. As Fig.\ref{fig:abs_gap} shows, only a small number of instances have absolute gap less than 1, so there is a large room for improvement in the quality of solution obtained by CG. In the example of $K=50$, the absolute gap of nearly $70\%$ of instances is greater than or equal to 1. In addition, as $K$ increases, the proportion of such instances gradually increases.

\begin{figure}[h]
  \centering
\includegraphics[width=0.4\textwidth]{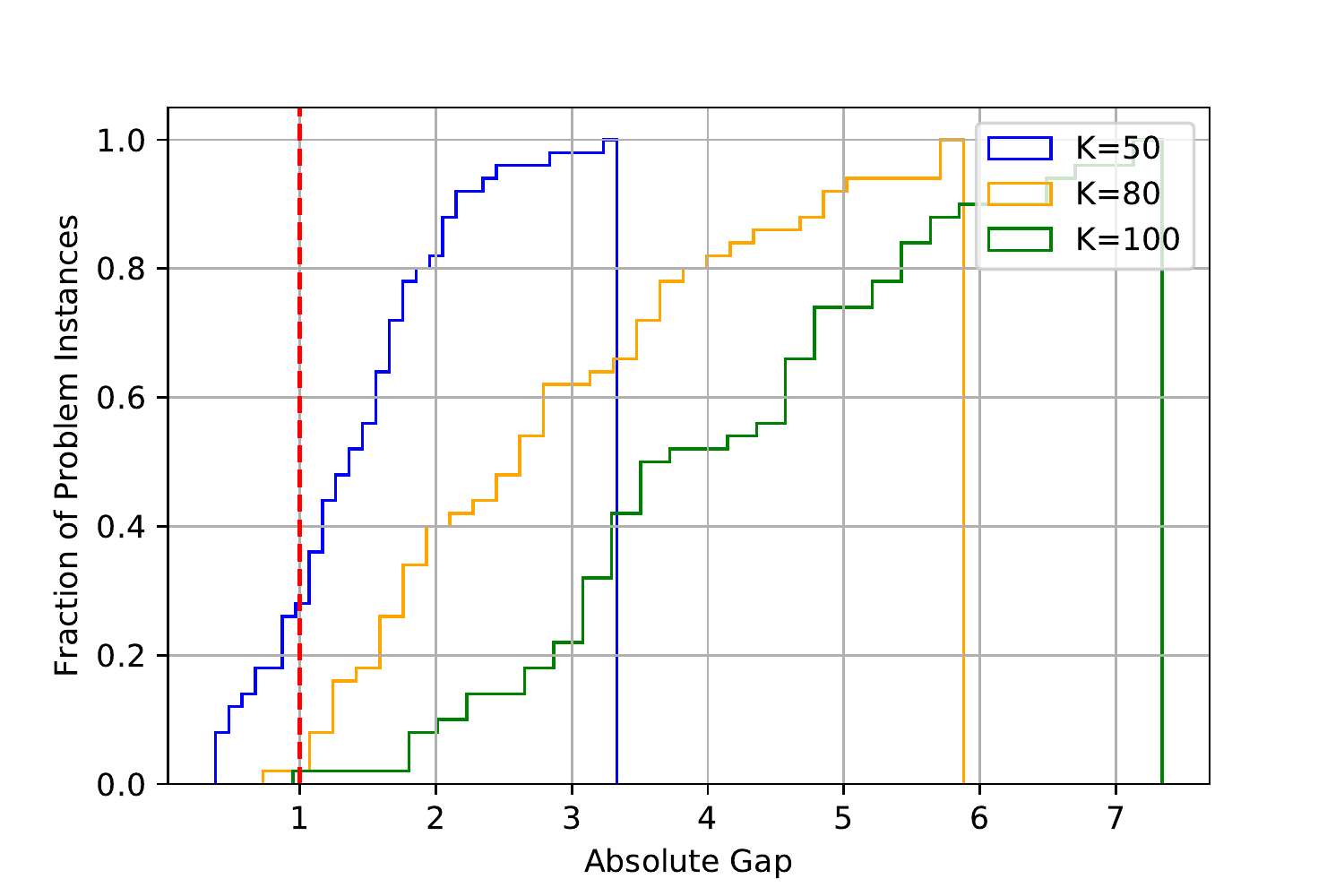}
  \caption{Empirical distribution of absolute gap of CG}
  \label{fig:abs_gap}
\end{figure}

The main reason leading to performance loss is that CG only enumerates part of the columns to ensure that the LP relaxation of problem (\ref{csp}) converges to its optimal solution. However, the columns obtained by CG may be not sufficient for solving the optimal solution of original problem. As shown in proposition \ref{bound}, we just obtain a feasible integer solution. In order to get the optimal solution of the original problem, we need to add additional columns.

\subsection{Motivation Ideas}
\label{sub:discussion}

The above two issues motivated our work in this paper. In retrospect, however, these two issues are internal to traditional CG itself: the solution quality is introduced by the relaxation from integer to linear programming, while the time cost is decided by the increasing difficulty of searching optimal solution of the RMP. 

Our idea is to enforce the basic CG with ensembles. To complement the relaxation gap, we generate columns from diverse solving paths and add high-quality columns even they are sub-optimal with respect to the RMP during CG. To accelerate the CG iterations, we actively add high-quality columns generated during CG iterations, which are sub-optimal to current RMP problem but found necessary to the final the column set. In this next section, we will give our ensemble approach in details.



\section{Ensembles of Column Generation}
\label{sec:methodology}
In this section, we detailed two ensemble CG method: \mc and \mt to address the problem of excessively long solving time and solution quality loss caused by traditional CG. Then, we present two techniques -- column selection and early stop, for further performance improvement. Besides, we give the theoretic analysis on optimality of our ensemble approach.

\subsection{\mc (Multi-Column) Method}
Column diversity is a key factor affecting algorithm performance. Since the columns obtained by CG is not sufficient to solve problem (\ref{csp}), an intuitive idea is to add additional columns to improve column diversity. 

CG starts from a limited fraction of feasible columns, and each iteration obtains a new column by solving subproblem. Instead of returning only the optimal column, the \mc method also returns $n$ sub-optimal columns with negative reduced costs to the RMP. The column generation process stops until the reduced costs of all columns are non-negative. Then, we can use the existing columns to solve problem (\ref{csp}) to get an integer solution.

The framework of \mc based CG is detailed in Algorithm \ref{mc-cg}. Note that unlike traditional CG, the $SubSub$ procedure (line 4, where $p_i$ is the returned column and $c_i$ its corresponding objective) returns multiple feasible columns. The main ideas of the algorithm is as follows:
\begin{itemize}
    \item In each iteration, by adding sub-optimal columns, the column diversity can be improved under the premise of ensuring LP convergence, thus improving the solution quality.
    \item Adding some sub-optimal columns in advance may reduce the total number of iterations of CG and accelerate the convergence process.
\end{itemize}

\begin{algorithm}
\DontPrintSemicolon
\caption{\mc Based CG Algorithm} 
\label{mc-cg}

\SetAlgoVlined
\SetKwInOut{Input}{Input}
\SetKwInOut{Output}{Output}
\SetKw{KwBy}{by}
\SetKwRepeat{Do}{do}{while}%

\Input{Parameters of Equation~\ref{csp} and \ref{eq:column_linear} (or \ref{eq:column_nonlinear}).}
\Output{Feasible solution $x_{p}, p \in \mathcal P'$ to Equation~\ref{csp}.}
\sf{
\Begin{
\nl \Do{$\max(\{c_i\})<0$}{
\tcp{Initialize the column set}
\nl $\mathcal P^{'} \gets$ ColumnInit($\{v_k\},V$)

\tcp{Solve the dual RMP}
\nl $\bar{\pi} \gets$ SolveDual ($\mathcal P^{'},\{b_k\}$)

\tcp{Solve and return multiple columns}
\nl $\{(p_i, c_i)\} \gets$ SolveSub ($\{v_k\},\bar{\pi}, V$)\\

\ForEach{$(p_i, c_i$)}{
    \nl \If {$c_i<0$ and $p_i$ is not in $\mathcal P^{'}$} {add $p_i$ to $\mathcal P^{'}$ }
}
}

\tcp{Solve Integer Programming}
SolveCSP ($\{b_k\},\mathcal P^{'}$)
}
}
\end{algorithm}
\subsection{\mt (Multi-Path) Method}

CG starts from a small number of feasible columns to search for other feasible columns with negative reduced cost. Different initial column sets lead to different search paths, which in turn leads to different final column sets. Therefore, another idea to improve column diversity is to search columns from different initial points and combine the columns found in different search paths to solve the problem (\ref{csp}).

\mt based CG consists of two stages: exploration and convergence. Hold all the columns in a base column set, and expand the set with new feasible columns found during each iteration. In the exploration stage, instead of letting CG iterate directly to convergence, after iteration $d$ step, we re-initialize the RMP by sampling a tiny fraction of columns from the current base column set. After that, search for feasible columns from the new starting point. Repeating the process $e$ times, the base column set contains columns searched from different starting points.

In the convergence stage, to ensure that the RMP converges to the optimal solution, we can search the remaining columns with negative reduced costs based on the current base column set. The framework of \mt based CG is detailed in Algorithm \ref{mt-cg}. Searching feasible columns from different starting points is conductive to increasing the diversity of columns and improving the solution quality.

\begin{algorithm}
\DontPrintSemicolon
\caption{\mt Based CG Algorithm} 
\label{mt-cg}

\SetAlgoVlined
\SetKwInOut{Input}{Input}
\SetKwInOut{Output}{Output}
\SetKw{KwBy}{by}

\Input{Parameters of Equation~\ref{csp} and \ref{eq:column_linear} (or \ref{eq:column_nonlinear}).}
\Input{The number of paths $e$ and maximum path depth $d$}
\Output{Feasible solution $x_{p}, p \in \mathcal P'$ to Equation~\ref{csp}.}

\sf{
\Begin{

\nl $\mathcal P^{'} \gets$ ColumnInit($\{v_k\},V$)

\tcp{Exploration Stage}
\nl \ForEach{$i=1 \rightarrow{e} $}{

\nl $ P_{s} \gets$ ColumnsSample ($\mathcal{P^{'}}$)

\nl \ForEach{$j=1 \rightarrow{d} $}{

\tcp{Solve the dual RMP}
\nl $\bar{\pi} \gets$ SolveDual ($\mathcal P^{'},\{b_k\}$)

\nl $(p,c) \gets$ SolveSub ($\{v_k\},\bar{\pi}, V$)\\

\nl \If {$c<0$} {add $p$ to $P_s$ and $\mathcal P^{'}$}
}
}

\tcp{Converge Stage}

\nl $\mathcal P^{'} \gets $ Conv($\{v_k\},b_k,V,\mathcal P^{'}$)

\tcp{Solve Integer Programming}
SolveCSP ($\{b_k\},\mathcal P^{'}$)
}
}
\end{algorithm}

\subsection{Implementation and Improvement}

\textbf{Implementation} 
Note that \mt and \mc can be combined in implementation. In Algorithms \ref{mc-cg} and \ref{mt-cg}, the subroutines $SolveDual$ and $SolveSub$ solve the optimization problems in standard Column Generation. For efficiency, we solve them via the solver Gurobi~\cite{gurobi}; note that other solvers or manual algorithms are possible. Details of our open sourced implementation is in appendix.

\textbf{Column Selection } 
Using the above two ensemble methods, we obtain a base column set (denoted as $\mathcal{P^{'}}$) with diverse columns at the end of the column generation process. In practice, this set may be very large and applying all columns to original problem will cause the integer programming difficult to solve. In order to filter the unnecessary columns without negative effect to the solution quality of original CSP problems, we further propose a column selection (CS) method to improve the solving efficiency.

For linear programming, the number of non-zero variables (the basic variables) in the optimal solution is equal to the number of constraints $K$. Following this insight, we develop a CS procedure by solving linear programming sequentially.

First, solve the RMP with all the columns in $\mathcal{P^{'}}$ and select the $K$ columns corresponding to the basic variables. Second, remove the selected columns from $\mathcal{P^{'}}$, solve the LP with $\mathcal{P^{'}}$ again and select a new set of column. Finally, repeating the selection process $s$ times, we obtain a subset of $\mathcal{P^{'}}$ and use these columns to solve integer programming. Removing the selected columns from $\mathcal{P^{'}}$ may result in the remaining columns not meet the feasibility condition. However, we can alleviate this problem quickly by constructing some auxiliary columns satisfying the feasibility criteria. 

\begin{algorithm}
\DontPrintSemicolon
\caption{Column Selection Procedure} 
\label{column selection}
\SetAlgoVlined
\SetKwInOut{Input}{Input}
\SetKwInOut{Output}{Output}
\SetKw{KwBy}{by}
\Input{Parameters of Equation~\ref{csp} and \ref{eq:column_linear} (or \ref{eq:column_nonlinear}).}
\Output{Pattern Set $\mathcal S \subset \mathcal{P^{'}}$}
\sf{
\Begin{
\nl \ForEach{$i=1 \rightarrow{s} $}{

\nl $ \mathcal{P^{'}} \gets $FeasibilityCheck $(\mathcal{P^{'}})$  

\nl $x_p \gets$ SolveRMP ($ \mathcal{P^{'}}, \{b_k\}$)

\nl \If {$x_p>0$} {$\mathcal S \gets \mathcal S \cup \{p\}$
$\mathcal {P^{'}} \gets \mathcal {P^{'}} \setminus \mathcal  \{p\}$}}}}
\end{algorithm}

The CS procedure is detailed in Algorithm \ref{column selection} and the intuition is as follows:
\begin{itemize}
    \item Compared to integer programming, linear programming can be solved efficiently. Therefore, the column selection process does not incur much computational cost. 
    \item RMP is the LP relaxation of corresponding integer programming with column set $\mathcal {P^{'}}$, the optimal column to the RMP may also be an optimal column to the integer programming.
\end{itemize}
\textbf{Early Stopping }
As shown in Fig.\ref{fig:rt_baseline}, for some problem instances, later iterations of CG are very time-consuming. Meanwhile, the columns found by the late search, while guaranteeing RMP convergence, may not be useful for integer programming. For example, as shown in Fig.~\ref{fig:dual_obj_bl}, the change value of dual objective values are very small, and we can relax the convergence to be no more than a small enough objective change. Therefore, we develop an early stopping procedure to improve the computation time.

The basic intuition is that as the CG approaches convergence, each new column entry provides less and less improvement over the RMP objective value. Therefore, we can monitor the improvement of the RMP objective value during the iteration process, and stop the CG when the improvement of the objective value is less than $\epsilon$ for $t$ consecutive times, so as to avoid too long iteration time. 

Since each iteration has to solve RMP to get the obtain dual variable, monitor the RMP objective value and check the stopping criterion does not impose additional computation burden.

\begin{figure}[h]
  \centering
\includegraphics[width=0.4\textwidth]{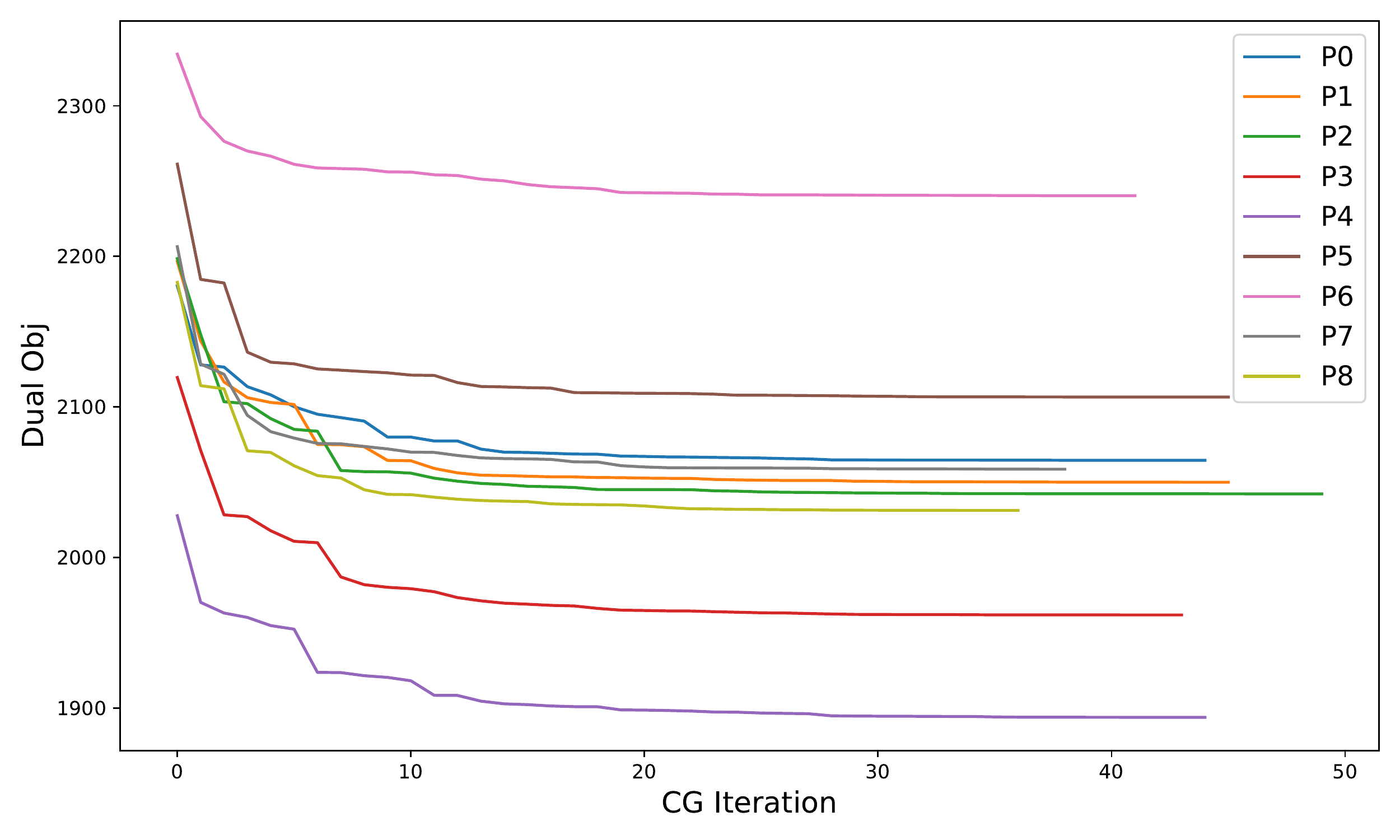}
  \caption{Dual Objectives of 9 problems of container scheduling with non-linear constraints}
  \label{fig:dual_obj_bl}
\end{figure}

\subsection{Theoretic Analysis}
\label{sub:theory}

We now analyze the convergence 
 and solution quality of the ensemble CG algorithm.

\begin{proposition}
\label{cg converge}
Assume that the feasible domain is non-empty and bounded, and every basic feasible solution of MP is nondenegerate, ensemble CG will converge to the optimal solution of MP in a finite number of steps.
\end{proposition}

\begin{proof}[Proof of Proposition~\ref{cg converge}]
According to the optimal condition of linear programming, if the reduced cost vector $\boldsymbol{c}$ satisfies that $\boldsymbol{c} \geq 0$, then the corresponding basic feasible solution $\boldsymbol{x}$ is optimal \cite{bertsimas1997introduction}. Since the feasible domain is non-empty and bounded, and every basic feasible solution of MP is nondenegerate, CG obtain a feasible column with negative reduced cost by solving subproblem and add it to RMP. The algorithm converges to the optimal solution of MP when the reduced cost of all columns is non-negative. Since the feasible column is limited, CG will stop in finite steps. Two ensemble CG methods keep the above characteristics. For \mc method, similar to CG, the algorithm converges when the reduced cost of all columns is non-negative. For \mt method,since the number of re-initialization and iteration is finite, the exploration stage stops in a finite number of steps. In the convergence stage, the algorithm adds the remaining columns with negative reduced cost and converges in a finite number of steps.
\end{proof}

\begin{proposition}
\label{bound}
Assume that the feasible domain of problem \ref{csp} is non-empty and bounded, then we can get the following bound relation:

\begin{equation}
     Z_{RMP}=Z_{MP}\leq Z^{*}_{IP}\leq Z_{IRMP}
\end{equation}

where $Z_{RMP}$ denotes the optimal value of the RMP,  $Z_{MP}$ is the optimal value of MP, $Z^{*}_{IP}$ represents the optimal value of the original problem~(\ref{csp}) and $Z_{IRMP}$ is the solution obtained by ensemble CG.
\end{proposition}

\begin{remark}
Proposition~\ref{bound} shows that the optimal value of MP and IRMP provides a lower bound and upper bound on the optimal value of the original problem~(\ref{csp}), respectively. The optimal solution to IRMP is a feasible solution to problem~(\ref{csp}). Ensemble CG aims to reduce the integral gap between problem(~\ref{csp}) and IRMP by adding various columns into IRMP. 
We can ensure that the integral gap is bounded by $Z_{IRMP}^*$-$Z_{MP}^*$.
\end{remark}


\section{Experimental Evaluation}
\label{sec:experiments}
In this section, we assess the performance of the proposed Ensemble CG algorithm on Cutting Stock Problems (CSP) with both linear and nonlinear cutting constraints. Our test problems are based on container scheduling in cloud computing, where the container size is an integer in the case of linear cutting constraints, and a random variable with a Gaussian distribution in the case of nonlinear cutting constraints. Compared to the traditional CG algorithm, our proposed method significantly improves the ratio of solving to optimality for linear CSP and reduces the solving time for nonlinear CSP while maintaining or enhancing the solution quality. 

\subsection{CSP with Linear Cutting Constraints}
\label{sub:exp-linear}

\textbf{Problem description: } We consider the container scheduling problem where container size is an integer number. Suppose there are $K$ services, each service has $b_k$ homogeneous containers with CPU size $v_k$. On cluster side, the CPU capacity of each node is $V$. The CSP is then to minimize the number of nodes that are used to hold all containers under capacity constraints. The corresponding formulation is Eq.\ref{csp} and Eq.\ref{eq:column_linear}. 

\textbf{Evaluation on scaling $K$}
We compare the algorithm performance for different problem size $K$. We detail the experiment results with $K$ ranging from $20$ to $60$, see appendix \ref{exp re} for the results of other scales. We take $50$ test cases for each $K$ values. Our compared methods are baseline, ensemble of \mc method (\#column$=3$), ensemble of \mt method (\#path $=10$), and ensemble of combined \mc-\mt method (\#column$=3$ and \#path $=10$).
The solution quality comparison and runtime are detailed in Fig.\ref{solution and gap} and Table \ref{tab:linearcase2 time}.

In Fig. \ref{fig:mc solution} and \ref{fig:mt solution}, "win", "tie" and "lose" represents solutions of our methods better, equal and worse than baseline, respectively. We compare the number of test cases in the three categories. As shown in Fig. \ref{fig:mc solution} and \ref{fig:mt solution}, our Ensemble CG methods can significantly improve the solution quality. As $K$ increase, the solution quality of nearly 60\% of the test cases is improved. Although in a small number of cases, the solution quality is worse than baseline. 

For those cases where the absolute gap of baseline is greater than $1$, Fig \ref{fig: abs gap comparison} shows the improvement of ensemble CG on the absolute gap. Both \mc and \mt methods are effective at reducing the absolute gap. In particular, the \mt method performs better than \mc method, and the combined \mc-\mt method has the best solution quality. As shown in Table \ref{tab:linearcase2 time}, the average runtime of  Ensemble CG methods is slightly longer than baseline in most cases. However, for $K=60$, the average runtime of Ensemble CG is shorter.

\begin{figure*}
  \centering
  \begin{subfigure}{0.33\textwidth}
\centering
\includegraphics[width=1.0\textwidth]{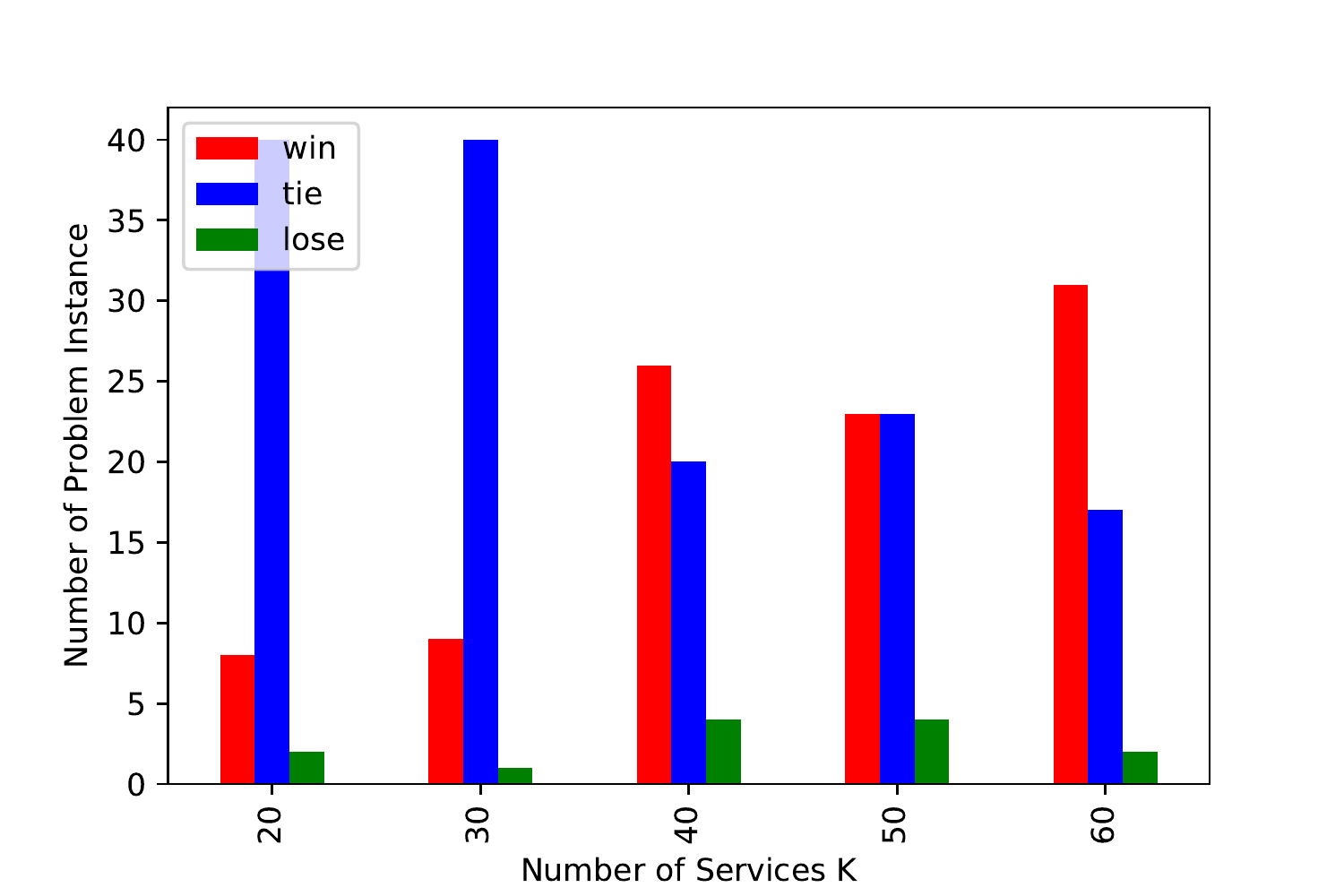}
  \caption{\mc solution quality}
  \label{fig:mc solution}
\end{subfigure}%
\begin{subfigure}{0.33\textwidth}
  \centering
\includegraphics[width=1.0\textwidth]{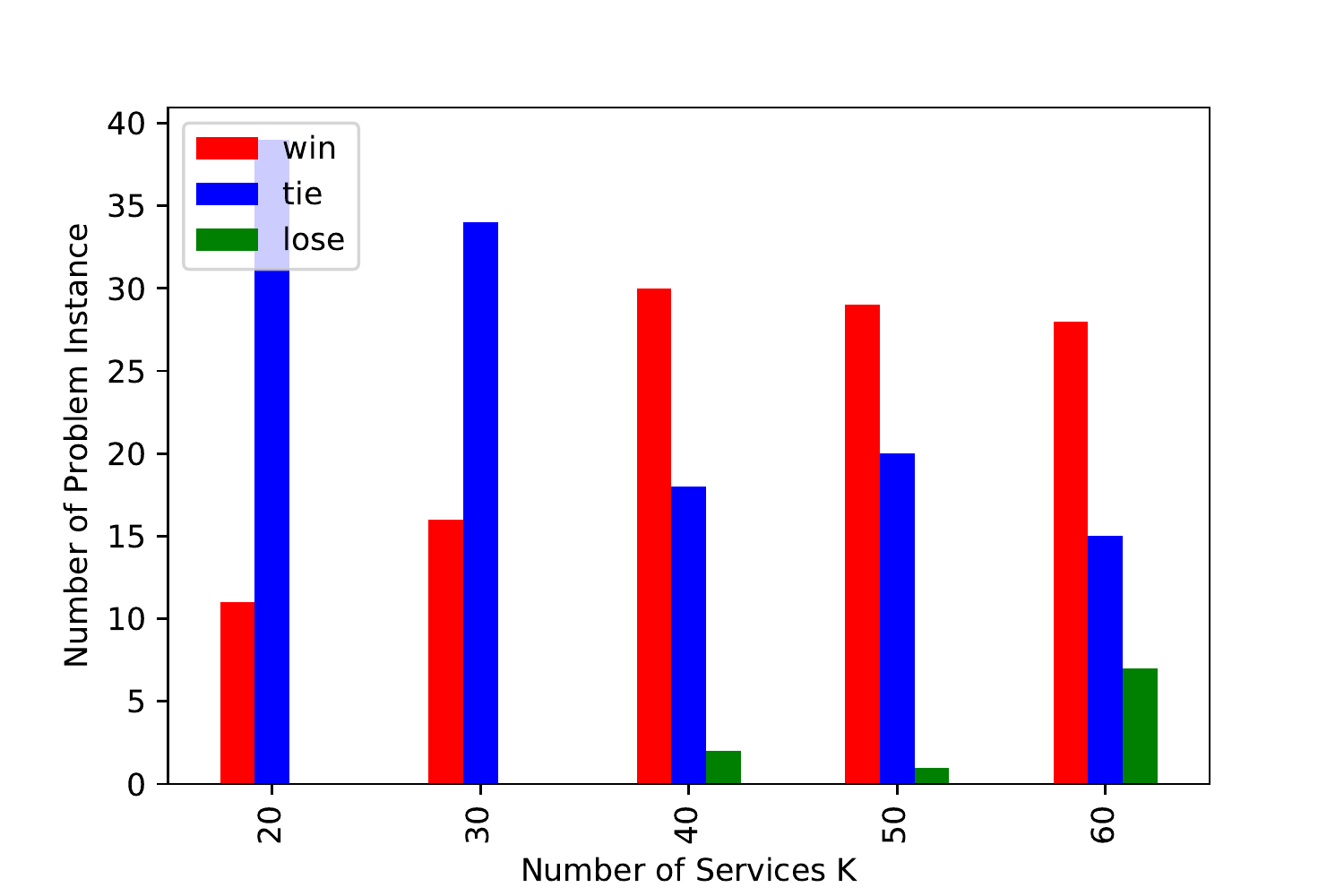}
  \caption{\mt solution quality}
  \label{fig:mt solution}
\end{subfigure}%
\begin{subfigure}{0.33\textwidth}
  \centering
\includegraphics[width=1.0\textwidth]{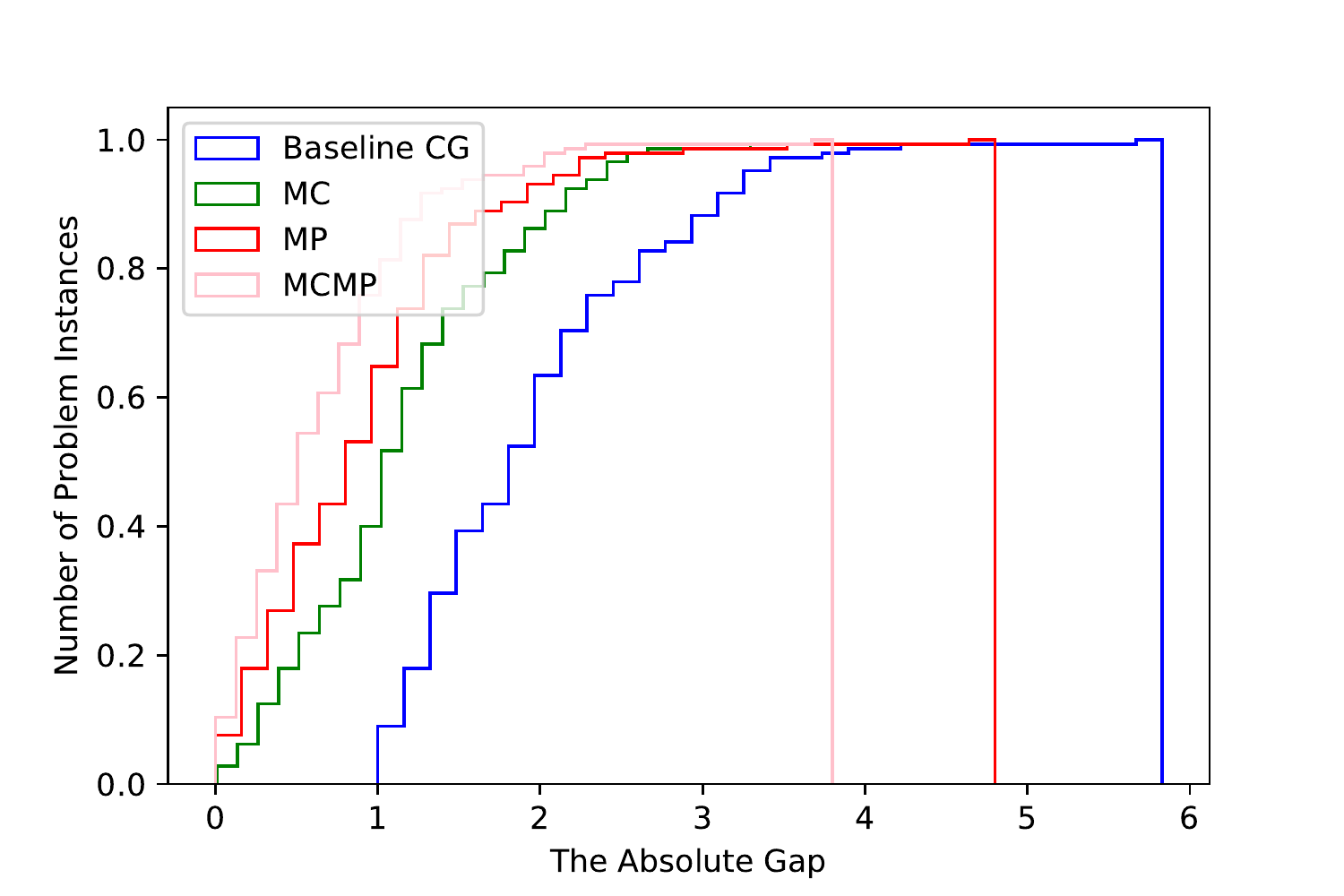}
  \caption{Solution absolute gap}
  \label{fig: abs gap comparison}
  \end{subfigure}
  \caption{Solution quality and absolute gap}
  \label{solution and gap}
\end{figure*}

\begin{table}
\caption{Runtime(s) in scaling K, averaged on 50 cases per k.}
\label{tab:linearcase2 time}
\scalebox{0.9} 
{
    \centering
    \begin{tabular}{r|r|r|r|r|}
    \toprule
        K & Baseline & \mc & \mt & MCMP\\
        \midrule
        20 &0.19&0.34&0.75&2.05\\
        30 &0.38&0.78&1.43&3.6\\
        40 &0.67&1.02&1.78&4.33\\
        50 &2.03&1.56&2.18&4.81\\
        60 &6.03&4.79&2.77&5.67\\
    \bottomrule
    \end{tabular}
}
\end{table}

\textbf{Ablation analysis }We analyze the benefits of our methods compared to baseline. The total number of columns generated and the convergence process are shown in Table \ref{tab:column num} and Fig.\ref{fig:converge_process}. As shown in Table \ref{tab:column num}, the Ensemble CG methods generate more columns than baseline, resulting in a higher diversity in the base column sets. The solutions obtained by columns selected from these sets are better. 
\begin{table}
\caption{Total number of columns generated, averaged on 50 cases per k.}
\label{tab:column num}
\scalebox{0.9} 
{
    \centering
    \begin{tabular}{r|r|r|r|r|}
    \toprule
        K & Baseline &\mc&\mt &MCMP\\
        \midrule
        20 &55.46&72.44&121.20&209.72\\
        30 &83.02&113.06&206.30&344.14\\
        40 &109.06&146.56&248.22&431.72\\
        50 &133.46&179.02&275.22&483.36\\
        60 &157.12&214.24&300.62&531.00\\
    \bottomrule
    \end{tabular}
}
\end{table}

\begin{figure}[h]
  \centering
  \includegraphics[width=0.45\textwidth]{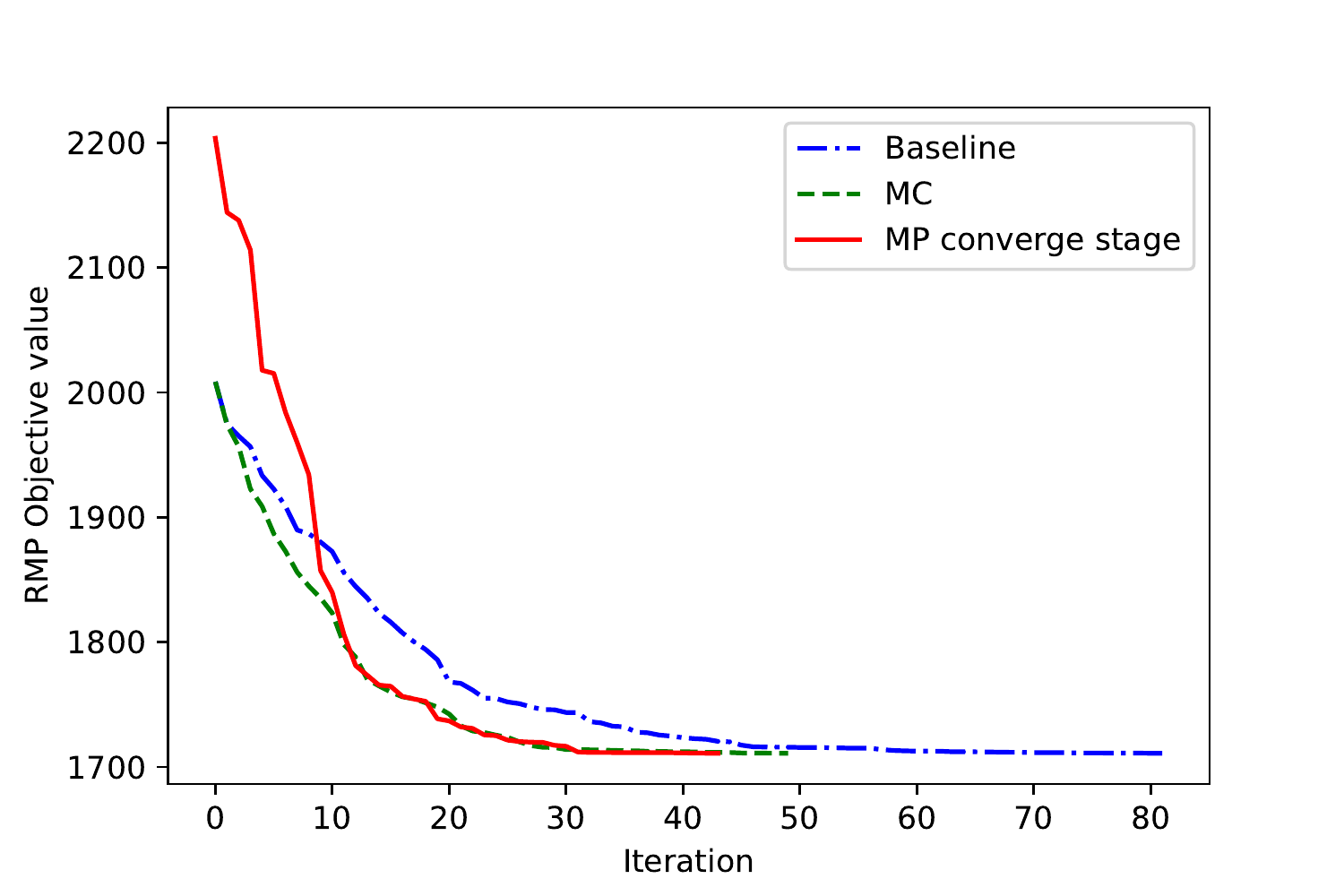}
  \caption{Convergence Process of Different Algorithm, K=60}
  \label{fig:converge_process}
\end{figure}

As shown in Fig \ref{fig:converge_process}, Ensemble CG methods converge faster than baseline. For \mc method, the total number of iterations is reduced because some sub-optimal columns are included in advance. For \mt method, since the exploration stage adds additional columns searched from different pathes, the convergence stage converges faster.

In our experiment, we set the column selection times $s$ equal to $6$ to prevent the IRMP from having too many columns and being difficult to solve. For linear CSP, the subproblem is solved quickly, so the results of adding early stopping are basically consistent with the above results.

\textbf{Hyperparameter analysis}
We test the effect of different values of hyperparameters on solution quality and runtime. We consider two hyperparameters: number of columns added per iteration for \mc and number of paths of \mt. The absolute gap and runtime are detailed in Fig.\ref{hyperparameters} and Table \ref{tab:hyperparam runtime}. As shown in Fig.\ref{hyperparameters}, the solution quality of Ensemble CG is better than baseline under different hyperparameter values. In particular, for \mc, a higher number of columns added in each iteration leads to a higher solution quality. For \mt, the more the number of search paths $e$, the better the solution quality. As for average runtime, with the increase of $n$, the runtime of \mc does not increase significantly. For \mt, however, the runtime increases as the $e$ increase. Therefore, we need to trade off the solution quality and runtime.

\begin{figure*}
\begin{subfigure}{0.5\textwidth}
  \centering
  \includegraphics[width=1.0\textwidth]{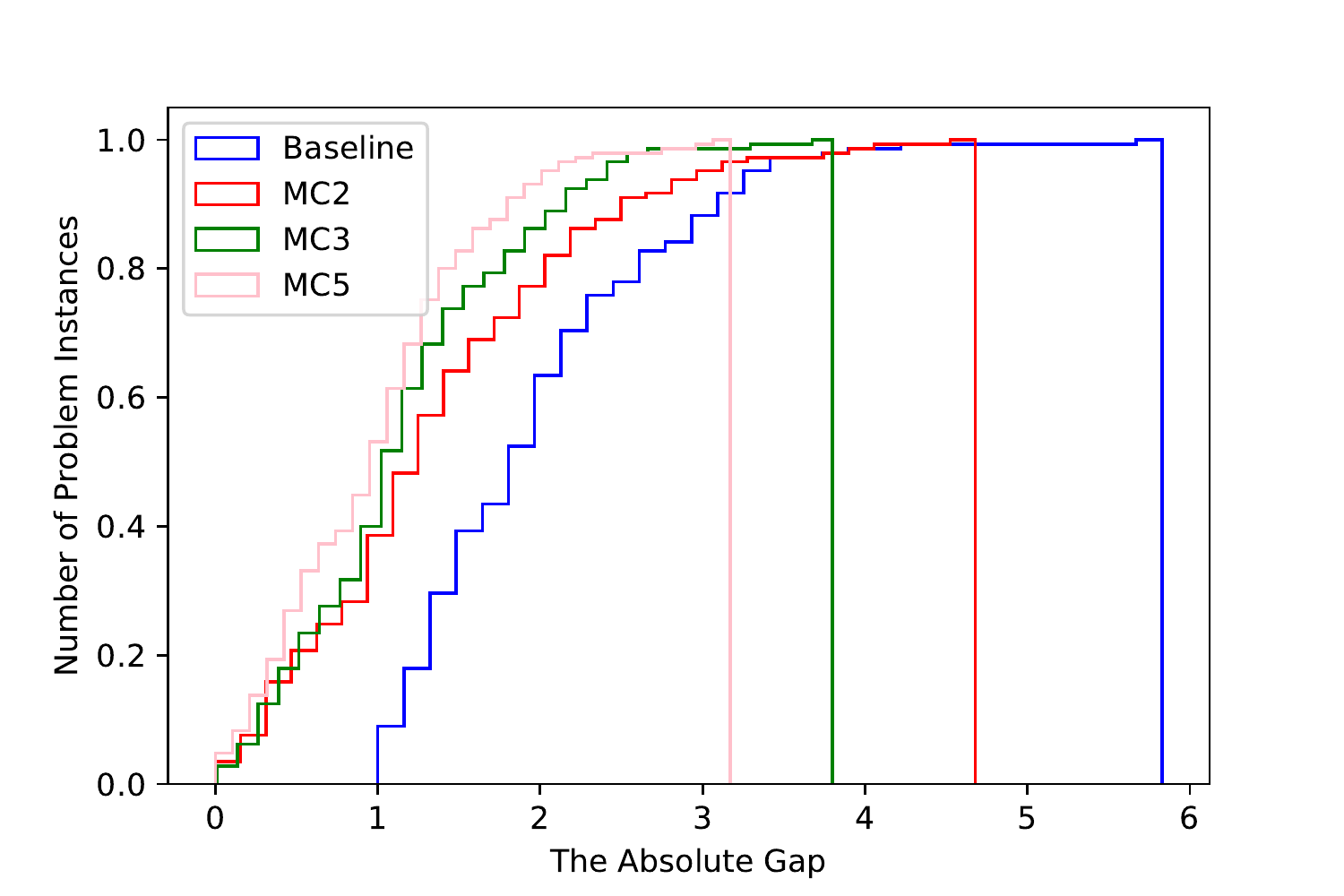}
  \caption{\#Columns added per iteration $n$}
  \label{fig:mc hyperparam}
\end{subfigure}%
\begin{subfigure}{0.5\textwidth}
  \centering
  \includegraphics[width=1.0\textwidth]{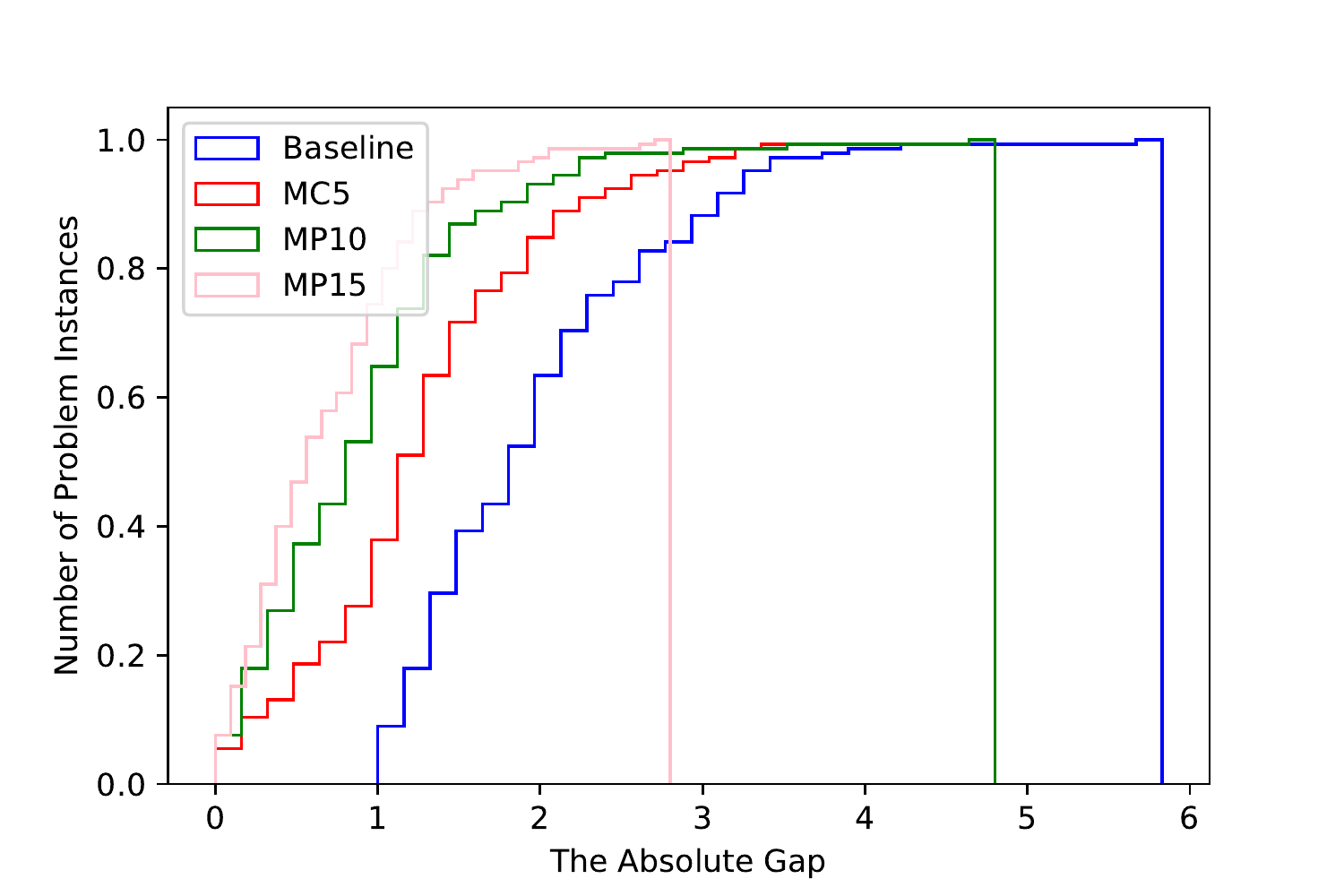}
  \caption{\#path for restart $e$}
  \label{fig:mt hyperparam}
\end{subfigure}
\caption{Effect of hyperparameters}
\label{hyperparameters}
\end{figure*}

\begin{table}
\caption{Runtime(s) in different hyperparameter values, averaged on 50 cases per k.}
\label{tab:hyperparam runtime}
\scalebox{0.9} 
{
    \centering
    \begin{tabular}{r|r|r|r|r|r|r|}
    \toprule
        K & \mc2  & \mc3 & \mc5  & \mt5 & \mt10 & \mt15\\
        \midrule
        20 &0.37&0.34&0.32&0.4&0.75&1.09\\
        30 &0.72&0.78&0.71&0.77&1.43&2.13\\
        40 &1.13&1.02&1.07&1.01&1.78&2.66\\
        50 &1.67&1.56&1.41&1.46&2.18&3.15\\
        60 &2.29&4.79&2.41&2.09&2.77&3.56\\
    \bottomrule
    \end{tabular}
}
\end{table}

\subsection{CSP with Nonlinear Cutting Constraints}
\label{sub:exp-nonlinear}

\textbf{Problem description: } We consider another container scheduling problem where container size is random variable following a Gaussian distribution, i.e., $v_k \sim \mathcal{N}(\mu_k, \sigma_k)$. Other settings are same with the problem in section \ref{sub:exp-linear}. In reality, this problem setting means all containers within a node share the CPU capacity under the chance constraint with confidence $\alpha$. For the specific details of the problem, we recommend readers to refer to the literature~\cite{yan2022solving}. The corresponding formulation is Eq.\ref{csp} and Eq.\ref{eq:column_nonlinear}.

\textbf{Evaluation on scaling $K$ } In practice, the number of services per cluster, such as K8S, varies from 5 to 15, and rarely exceeds 20 in our experience. Hence, we conduct tests with 10 to 20 services, using synthetic container sizes and quantities. For each $K$ value, we test 9 cases. The methods we compare are the baseline, an ensemble of the \mc method (with 3 columns), an ensemble of the \mt method (with 4 paths), and an ensemble of the combined \mc-\mt (with 3 columns and 4 paths). The runtime and solution quality, measured as the solve-to-optimal ratio, are summarized in Tables~\ref{tab:rtm-nonlinear} and~\ref{tab:perf-nonlinear} respectively. As seen in Table~\ref{tab:rtm-nonlinear}, our Ensemble methods can significantly reduce the runtime compared to the baseline CG. Specifically, for small-scale problems, the \mt method may have a higher time cost than the baseline, while for large-scale problems, both \mc and \mt significantly reduce the solving time. As seen in Table~\ref{tab:perf-nonlinear}, our Ensemble methods show the same or better solution quality (higher solve-to-optimal ratios) compared to the baseline, although the improvement is slight as the baseline is already close to optimal. Furthermore, we found that using an early stopping improvement, the runtime can be further reduced without sacrificing solution quality, as shown in the tables (E4C3+ES, where early stopping occurs after 4 iterations of exceeding the dual objective by 0.05).

\begin{table}
\caption{Runtime(s) in scaling K, averaged on 9 cases per k.}
\label{tab:rtm-nonlinear}
\scalebox{0.9} 
{
    \centering
    \begin{tabular}{c|r|r|r|r|r|}
    \toprule
        K & Baseline & MC3 & MP4 & MP4MC3 & MP4MC3+ES\\
        \midrule
        10 & 12.5 & $\hlc{8.1}$ & 25.0 & 20.8 & - \\
        15 & 259.9 & $\hlc{169.5}$ & 386.2 & 358.1 & 358.8\\
        20 & 2856.7 & $\hlc{1758.4}$ & $2826.7$ & $\hlc{2011.1}$ & $\hlc{1901.7}$\\
    \bottomrule
    \end{tabular}
}
\end{table}

\begin{table}
\caption{Optimal ratio in scaling K with 9 cases per k.}
\label{tab:perf-nonlinear}
\scalebox{0.9} 
{
    \centering
    \begin{tabular}{c|r|r|r|r|r|}
    \toprule
        K & Baseline & MC3 & MP4 & MP4MC3 & MP4MC3+ES\\
        \midrule
        10 & 8/9 & $\hlc{9/9}$ & 8/9 & $\hlc{9/9}$ & -\\
        15 & 8/9 & 8/9 & 8/9 & $\hlc{9/9}$ & $\hlc{9/9}$\\
        20 & 8/9 & $\hlc{9/9}$ & $\hlc{9/9}$ & $\hlc{9/9}$ & $\hlc{9/9}$\\
    \bottomrule
    \end{tabular}
}
\end{table}

\begin{figure}[h]
  \centering
  \includegraphics[width=0.45\textwidth]{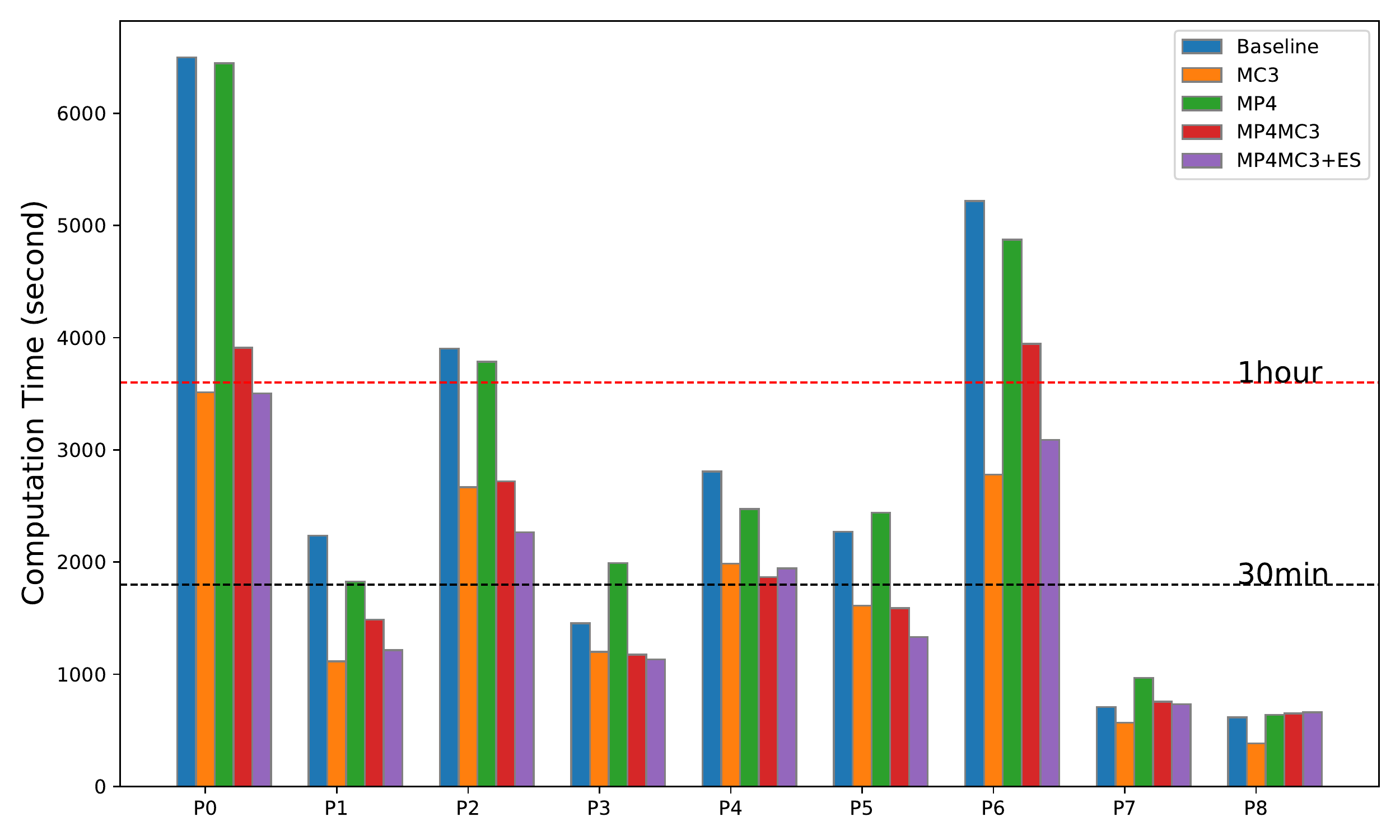}
  \caption{Comparison of runtime at K=20}
  \label{fig:rtm_k20}
\end{figure}

\begin{figure}[h]
  \centering
  \includegraphics[width=0.45\textwidth]{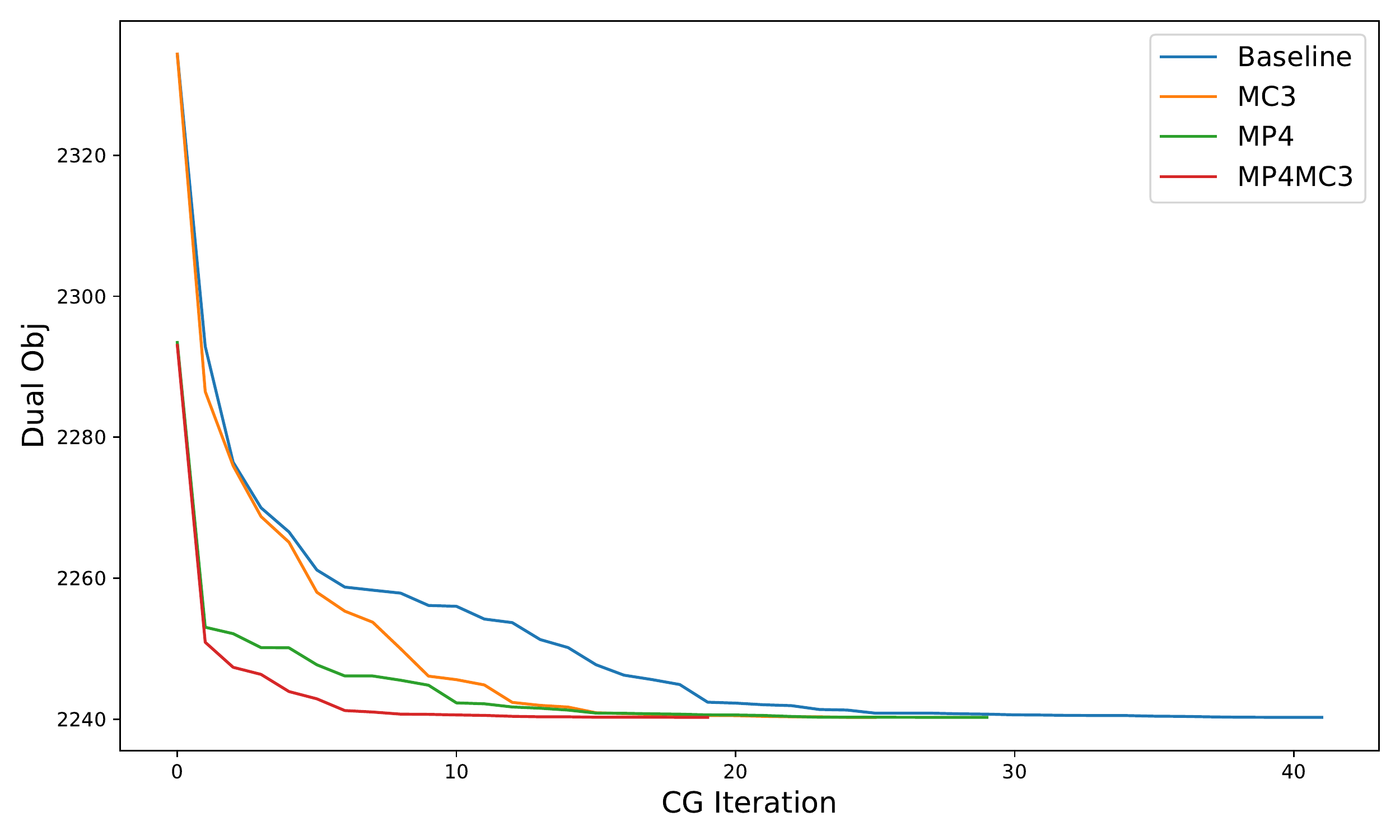}
  \caption{Convergence Processes at K=20}
  \label{fig:nonlinear_conv_k20}
\end{figure}

Now we look into the detailed run time per method on each of 9 problem cases at $K=20$ (same settings as Fig.\ref{fig:rt_baseline} and \ref{fig:dual_obj_bl}). As shown in Fig.\ref{fig:rtm_k20}, effect of ensemble varies on different cases. Even so, the qualitative effect of all methods is relatively consistent across different cases.

Finally, we investigate the objective converge processes of all methods, still taking the 9 cases of $K=20$ as examples. For comparability, we adopt all iterations of baseline and \mc, and the final path of iterations in \mt and \mt-\mc combination. The results are plotted in Fig.\ref{fig:nonlinear_conv_k20}. As shown, similar to linear cases, ensemble methods converge faster than baseline. (Note that the \mt method has to search for more short pathes of patterns before the final path, which requires extra time cost). For clearness, we did not depict the method of \mt-\mc with early stopping in the figure, which is a right-hand truncated curve of \mt-\mc.

\section{Related Work}
\label{sec:related_work}

\textbf{Cutting Stock Problem}
The Cutting Stock Problem, identified by Kantorovich~\cite{kantorovich1960mathematical}, is one of the oldest and most classic problems in combinatorial optimization. It has been widely used in different fields\cite{oliveira2021extended, parreno2021mathematical} since it was put forward. However, the CSP is an NP-hard problem, which is very difficult to solve. Gilmore and Gomory~\cite{gilmore1963csp2} proposed the CG algorithm to solve the linear relaxation of the reformulation of CSP.  It is well known that the optimal solution of the LP relaxation of the Gilmore Gomory model provides a tight lower bound on the optimal number of paper rolls~\cite{ben2005cutting} . In fact, for most one-dimensional CSP, the difference between the optimal value of integer programming and its continuous relaxation is less than 1. Rietz and Scheithauer~\cite{rietz2002tighter} constructed some problem instances with 7/6 intervals. 

In practice, motivated by real-time requirements, some approximate algorithms such as Best Fit and First Fit are usually used to obtain approximate solutions of CSP. There are also heuristics or meta-heuristics used to solve CSP problems see \cite{cerqueira2021modified}and \cite{sadykov2019primal}. For a detailed model and algorithm of CSP, we recommend this review~\cite{delorme2016bin}. 

\textbf{CSP with nonlinear constraints} In reality, the constraints of CSP may be nonlinear. Wang et al.\cite{wang2020two} consider a variant of CSP and use column-and-row generation framework to solve of LP relaxation of it. The framework involves a nonlinear IP subproblem and they exploit the linearization technique and propose a decomposition method to deal with it. Our problem includes a nonlinear includes a nonlinear constraint due to the gaussian transformation of chance constraint. Since this constraint only appears in subproblem and gurobi supports non-convex quadratic constraints, we can use gurobi to solve subproblem without affecting the MP.

\textbf{Column Generation Algorithm} Column Generation is a classical and efficient algorithm for solving large-scale LP  and also widely used in solving practical integer programming problems \cite{porumbel2017constraint, lan2022column, pan2023demand}.
 To obtain accurate integer solution, CG can be embedded into the framework of branch and bound algorithm, known as branch-and-price algorithm. \cite{barnhart1998branch,wei2020new, alkaabneh2022multi}.

Recently, there has been an increasing amount of research interested in using machine learning algorithms to improve the performance of CG. Convergence speed is one of the bottlenecks of CG. Traditional CG converges slowly and end up adding a lot of columns \cite{pessoa2018automation}. 
Morabit et al.~\cite{morabit2021machine} applied a learning model to the column selection problem, choosing the best column from a column set at each iteration. This method improves the CG solution time by 30\%. Babaki et al.~\cite{babaki2022neural} formulated the problem of choosing which column to add as a sequential decision problem and proposed a neural CG architecture, which determines column addition by predicting the optimal dual variables. They used VRP task to prove their method performed well. Shen et al.~\cite{shen2023adaptive} describe the decision variables by using the feedback of heuristic search, so that the machine learning model can adaptively predict the optimal solution. Their method can be used to solve the subproblem of CG heuristically. Besides our work, multi-column, that is to add multiple columns of the RMP's feasible solutions, has been used in practical applications especially when the allowed solving time is limited; its effect to accelerate CG convergence is also mentioned in few literature \cite{morabit2021machine}. Other literature focus on the dual oscillation problem of CG, an unstable behavior that results in a longer number of iterations of CG. A lot of stabilization techniques have been proposed to solve this problem such as adding penalties so that the dual solution does not stray too far from the center\cite{lee2011chebyshev}, or using information from the previous dual solution to smooth the current solution\cite{pessoa2018automation}. 

\section{Conclusions}
In this paper, we re-investigated the classic column generation method for solving cutting stock problems, and identified the issues solution quality and time cost in applications with linear and nonlinear constraints emerged in approximating chance constraints. We propose the ensemble approach, including \mc and \mt methods, to realize these issues. Experiments on CSPs derived from container scheduling show that our methods can significantly reduce the time cost or improve the solution quality for problems with linear or nonlinear cutting constraints. In the future, we shall explore more methods of ensemble to achieve more benefits in larger scale CSP problems.

\begin{acks}
We thank MS colleagues for valuable discussion.
\end{acks}

\bibliographystyle{ACM-Reference-Format}
\bibliography{sigmetrics-ref}

\clearpage
\appendix

\section{Supplementary Content to Section \ref{sub:theory}}
\label{app:proofs}

\begin{proof}[Proof of Proposition~\ref{bound}]

Let $\mathcal{S}$ denotes the column set obtained by ensemble CG. $\mathcal{P}$ represents all feasible columns of problem (\ref{csp}). Then, we have that $\mathcal{S} \subset \mathcal{P}$ .By definition, we have the following equations: \\
$Z_{RMP}=min\{\sum_{p \in \mathcal S} x_p| \sum_{p \in \mathcal S} a_{k p} x_p \geq b_k, \forall k \in [K], x_p \geq 0, \forall p \in \mathcal S\}$, \\
$Z_{MP}=min\{\sum_{p \in \mathcal P} x_p| \sum_{p \in \mathcal P} a_{k p} x_p \geq b_k, \forall k \in [K], x_p \geq 0, \forall p \in \mathcal P\}$, \\
$Z^{*}_{IP}=min\{\sum_{p \in \mathcal P} x_p| \sum_{p \in \mathcal P} a_{k p} x_p \geq b_k, \forall k \in [K], x_p \in I, \forall p \in \mathcal P\}$, \\ 
$Z_{IRMP}=min\{\sum_{p \in \mathcal S} x_p| \sum_{p \in \mathcal S} a_{k p} x_p \geq b_k, \forall k \in [K], x_p \in I, \forall p \in \mathcal S\}$.

$Z_{RMP}=Z_{MP}$ is trivial.  Since we know that RMP converges to the optimal solution of MP when the CG process ternimates. The base column set $\mathcal P^{'}$ covers the columns corresponding to the optimal basis variables of the MP and the CS process will also keep these columns in $\mathcal S$. Therefore, the optimal value of these two LP problems are equal.

We now show that $Z_{MP}\leq Z^{*}_{IP}$. Since MP is the LP relaxation of problem~(\ref{csp})  (IP for short), we know that the feasible region of IP is a subset of the feasible region of MP, which means that:
\begin{equation*}
\begin{aligned}
\forall &\boldsymbol{x} \in \left\{\boldsymbol{x}~\bigg|\sum_{p \in \mathcal P} a_{k p} x_p \geq b_k, \forall k \in [K], x_p \in I, \forall p \in \mathcal P\right\} \\
\implies &\boldsymbol{x} \in \left\{\boldsymbol{x}~\bigg|\sum_{p \in \mathcal P} a_{k p} x_p \geq b_k, \forall k \in [K], x_p \geq 0, \forall p \in \mathcal P\right\}
\end{aligned}
\end{equation*}

Let $\boldsymbol{x^*}$ and $\boldsymbol{\hat{x}^*}$ denotes the optimal solution of RMP and IP respectively. By definition, we have
\begin{equation}
    Z_{MP} = \sum_{p \in \mathcal P} x^{*}_{p} \leq \sum_{p \in \mathcal P} \hat{x}^{*}_{p} = Z^{*}_{IP} 
\end{equation}

Finally, we prove that $Z^{*}_{IP}\leq Z_{IRMP}$. $\forall \boldsymbol{x} \in \{\boldsymbol{x}~|\sum_{p \in \mathcal S} a_{k p} x_p \geq b_k, \forall k \in [K], x_p \in I, \forall p \in \mathcal S\} $, by letting $x_p=0,\forall p \in \mathcal P \backslash \mathcal S$, we can always construct a feasible solution $\boldsymbol{\tilde{x}}=(\boldsymbol{x},x_p)$, which satisfies that $\boldsymbol{\tilde{x}} \in \{\boldsymbol{x}~|\sum_{p \in \mathcal P} a_{k p} x_p \geq b_k, \forall k \in [K], x_p \in I, \forall p \in \mathcal P\} $. Let ${x}^*$  represents the optimal solution of IRMP, and $\tilde{x}$ is the corresponding feasible solution of IP constructed by ${x}^*$, then we have
\begin{equation}
    Z^{*}_{IP}\leq \sum_{p \in \mathcal P} \tilde{x}_{p} = \sum_{p \in \mathcal S} {x}^{*}_{p}+\sum_{p \in \mathcal P\backslash \mathcal S} {x}_{p}
    =\sum_{p \in \mathcal S} {x}^{*}_{p}=Z_{IRMP}
\end{equation}
Therefore, we conclude that $Z_{RMP}=Z_{MP}\leq Z^{*}_{IP}\leq Z_{IRMP}$.
\end{proof}

\begin{proposition}
\label{gap}
If problem~(\ref{csp}) is feasible, let \\
$Z_{UB}=min\{\sum_{p \in \mathcal C} x_p| \sum_{p \in \mathcal C} a_{k p} x_p \geq b_k, \forall k \in [K], x_p \in I, \forall p \in \mathcal C\}$, $Z_{LB}=(\frac{\bar{\pi}}{\bar{\pi}A^{min}})b$,\\
then we can obtain the following bound relation:

\begin{equation}
    Z_{LB}\leq Z_{IP}^{*}\leq Z_{UB}
\end{equation}
where $C$ is the base column set, $\bar{\pi}$ is the optimal dual variables corresponding to $C$, $A^{min}$ is the optimal column find by $\bar{\pi}$.
\end{proposition}
\begin{proof}[Proof of Proposition~\ref{gap}]
According to \cite{farley1990note}, Farley's bound $Z_{LB}$ provides a lower bound on LP relaxation of problem~(\ref{csp}), therefore, we have $Z_{LB}\leq Z_{MP}$. From proposition2, we know that $Z_{MP} \leq Z_{IP}^{*}$, then we obtain that $Z_{LB}\leq Z_{IP}^{*}$. 

Since $\mathcal C$ denotes the base column set, we know that $\mathcal C\in \mathcal P$, similar to the proof of $Z_{IP}^{*} \leq Z_{IRMP}$, we get $Z_{IP}^{*} \leq Z_{UB}$.

Therefore, we conclude that $Z_{LB}\leq Z_{IP}^{*}\leq Z_{UB}$ and the $gap=Z_{UB}-Z_{LB}$ provides an upper bound on the distance between $Z_{IP}^{*}$ and $Z_{UB}$.
\end{proof}

\begin{remark}
    When earlying stopping is used, CG will not converge to the optimal solution of MP, and then we can use $Z_{LB}$ to bound the distance between $Z_{UB}$ and the optimal solution $Z_{IP}^{*}$.
\end{remark}

\section{Experimental details}
\label{app:experiemnts}

In this section, we introduce the problems used in our evaluation and hyper-parameters of our ensemble methods. We open source our code with default parameters at \href{https://github.com/AnnonymousAuthor/KDD398}{this repo}.

\subsection{Problem Settings}

We briefly introduce the problem settings used in evaluation. For more detials, please refer to our open sourced code.

\textbf{CSPs with linear cutting constraints} 
For different $K$, we randomly generate the container CPU sizes $v_k$ and numbers $b_k$. For each problem, the of the container sizes are random integer in $[1,70]$, the containers numbers are random integer in $[50,200]$. The CPU size $V$ of each node is $127.58$.

\textbf{CSPs with nonlinear cutting constraints } We have 30 base services with container sizes following the Gaussian distributions as shown in Table \ref{tab:nonlinear_synthetic_stats}. For each problem, we sample $K$ services from the pool without replacement. Besides, the node capacity $V=31.58$\footnote{This strange number is inherited from \cite{yan2022solving}, where the hardware has 32 cores but the host operating system requires $0.42$ cores for management work.}. 

\begin{table}
    \centering
    \scalebox{0.75} 
    {
    \begin{tabular}{c|cccccccccc}
        \toprule
        Means & 0.73& 0.84& 0.97& 1.01& 1.06& 1.06& 1.06& 1.07& 1.12& 1.17 \\
        Std & 1.73& 0.47& 0.43& 2.69& 0.85& 0.19& 0.9& 0.82& 0.97& 0.62 \\
        Number & 270& 55& 1618& 904& 576& 1085& 1035& 118& 1450& 313 \\
        \midrule
        Means & 1.51& 1.52& 1.56& 1.57& 1.94& 1.96& 2.41& 2.42& 2.46& 2.47 \\
        Std & 0.31& 0.62& 0.84& 0.57& 0.7& 0.55& 0.9& 1.15& 1.95& 0.99 \\
        Number & 44& 544& 697& 427& 363& 360& 701& 1425& 305& 228 \\
        \midrule
        Means & 2.48& 2.48& 2.49& 2.52& 2.59& 3.33& 3.81& 4.12& 6.18& 6.97 \\
        Std & 0.66& 0.75& 0.25& 1.23& 1.21& 0.47& 0.87& 1.28& 0.46& 0.62 \\
        Number & 1552& 378& 606& 180& 293& 1424& 501& 1019& 1580& 405 \\  
        \bottomrule
    \end{tabular}
    }
    \caption{Service statistics of the nonlinear dataset.}
    \label{tab:nonlinear_synthetic_stats}
\end{table}

\subsection{Hyper-parameters of Ensemble CG}
The main hyper-parameters of our ensemble methods are as follows.
\begin{itemize}
\item num\_path: the number of paths in ensembles;
\item depth: the maximum length of paths in ensembles;
\item num\_column: the maximum number of columns returned once by solve subproblem;
\item dual\_objective\_threshold: the threshold of dual objective change for breaking;
\item dual\_throttles: the number of exceeding dual\_objective\_threshold for breaking; 
\item sub\_time\_limit: the time limit allowed to solve the subproblem;
\item sub\_throttles: the number of exceeding sub\_time\_limit for breaking;
\item csp\_time\_limit: the time limit allowed to solve the CSP problem;
\item selection\_times: the number of filtering the column set after CG;
\item convergence: whether strictly iterate to converge to a criteria.
\end{itemize}
\subsection{Linear experiment results for different K}
\label{exp re}
For linear CSP, we also test the performance of the algorithm at other sizes $K$. $K$ in experiment 1 is $5,10,15,20,30$, and $K$ in experiment 2 is $20, 30, 50, 80, 100$. The absolute gap is shown in the Fig. \ref{fig:exp1} and Fig.\ref{fig:exp2}. We can find that even for different $K$, the solution quality of our algorithm is still significantly better than the baseline.
\begin{figure}[h]
  \centering
\includegraphics[width=0.45\textwidth]{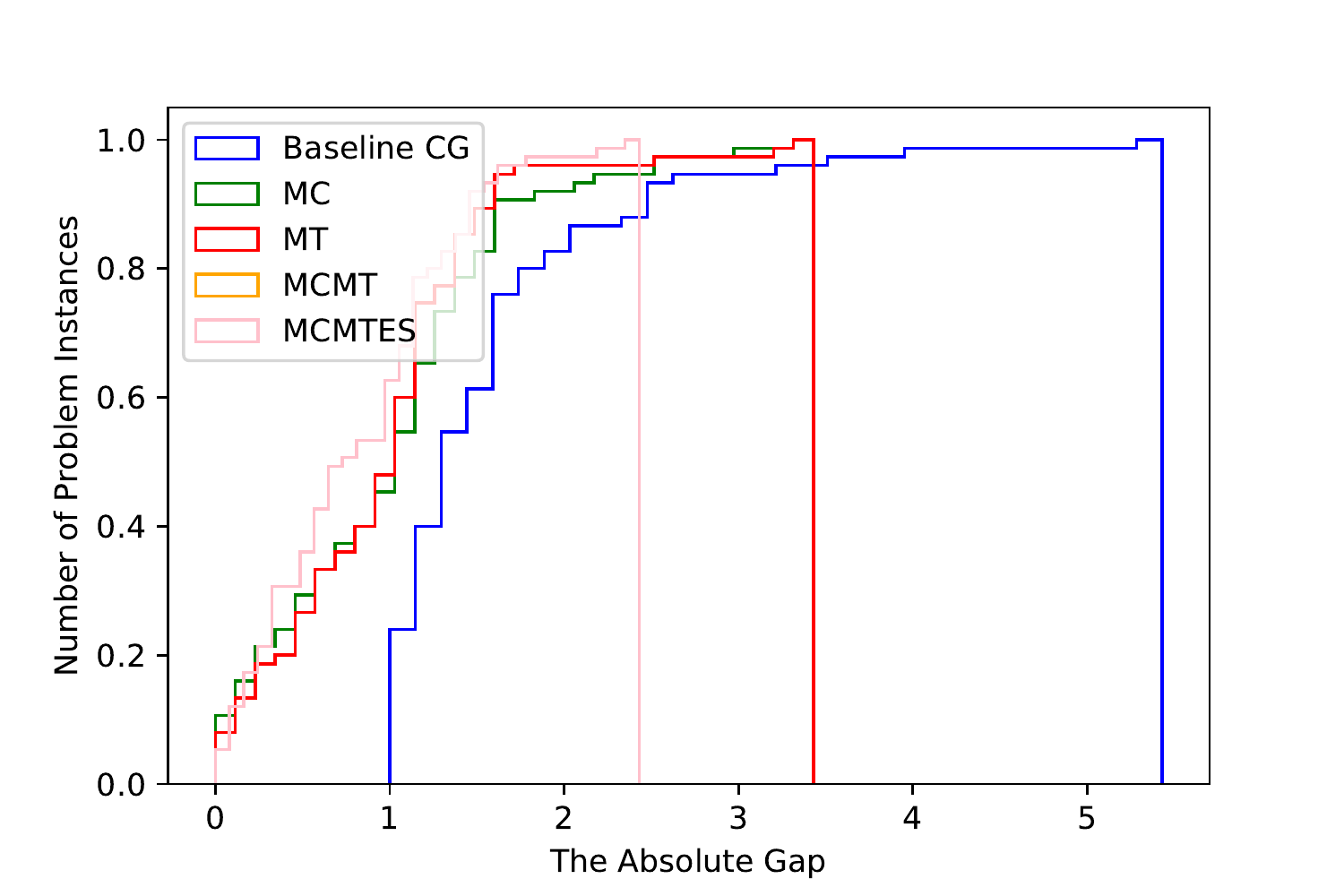}
  \caption{Absolute gap of experiment 1}
  \label{fig:exp1}
\end{figure}

\begin{figure}[h]
  \centering
  \includegraphics[width=0.45\textwidth]{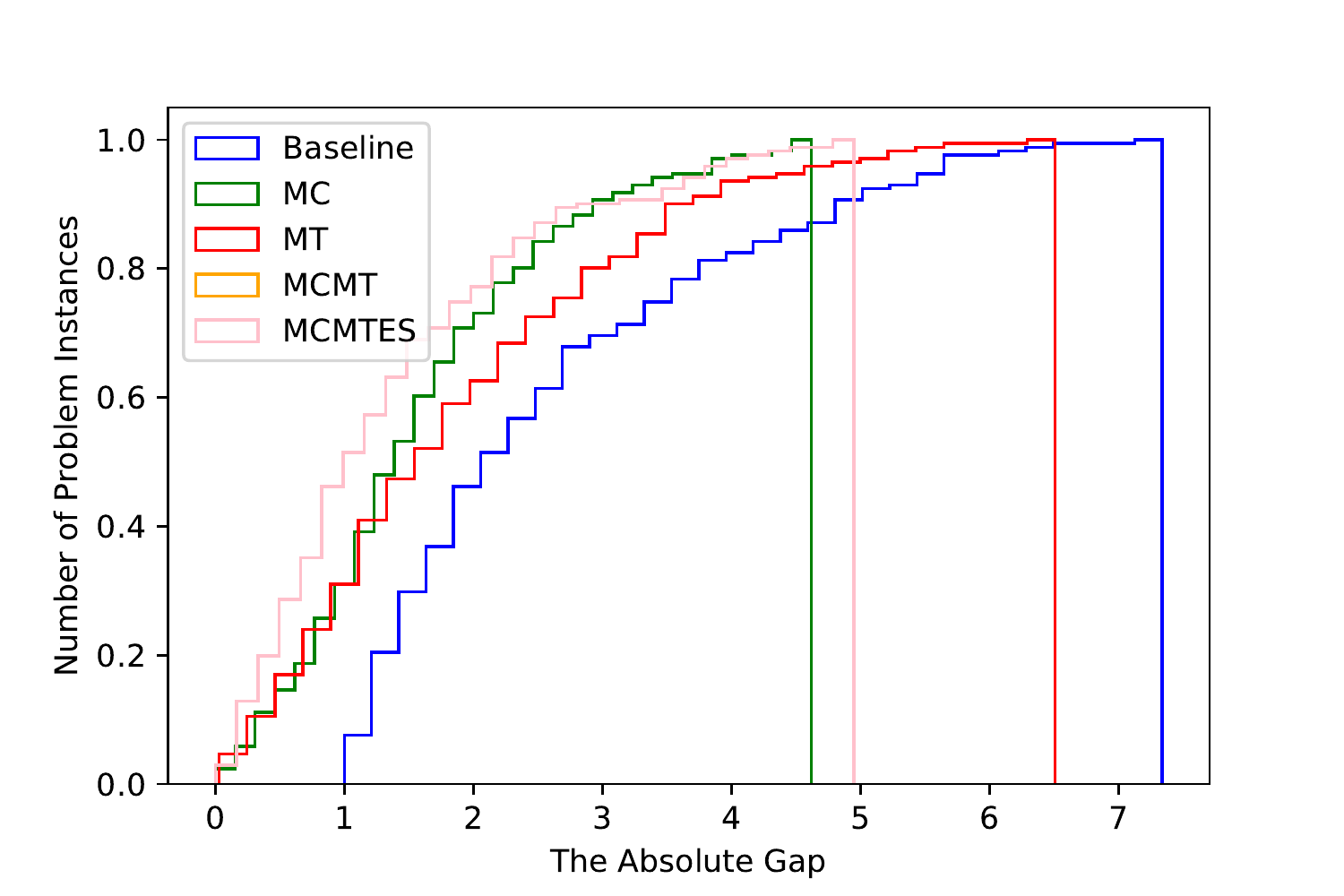}
  \caption{Absolute gap of experiment 2}
  \label{fig:exp2}
\end{figure}
\end{document}